\documentclass{amsart}

\usepackage{amsmath,amssymb,amsthm,mathtools}
\usepackage{enumitem}
\usepackage{mathrsfs}
\usepackage{xcolor}
\usepackage{aliascnt}
\definecolor{refblue}{RGB}{0,27,126}
\definecolor{citegreen}{RGB}{0,126,68}
\usepackage[
  colorlinks=true,
  linkcolor=refblue,
  citecolor=citegreen,
  urlcolor=purple,
  filecolor=purple]{hyperref}
\hypersetup{pdftitle={Sharp harmonic-mean inequalities for Neumann and Aharonov--Bohm spectra},pdfauthor={Daguang Chen, Chengxi Yang}}
\usepackage[nameinlink,capitalize,noabbrev]{cleveref}
\usepackage{microtype}
\usepackage[a4paper, left=2.8cm, right=2.8cm, top=2.5cm, bottom=2.5cm]{geometry}

\numberwithin{equation}{section}

\newtheorem{theorem}{Theorem}[section]

\newaliascnt{proposition}{theorem}
\newtheorem{proposition}[proposition]{Proposition}
\aliascntresetthe{proposition}

\newaliascnt{lemma}{theorem}
\newtheorem{lemma}[lemma]{Lemma}
\aliascntresetthe{lemma}

\newaliascnt{corollary}{theorem}
\newtheorem{corollary}[corollary]{Corollary}
\aliascntresetthe{corollary}

\theoremstyle{definition}
\newaliascnt{definition}{theorem}

\aliascntresetthe{definition}

\newaliascnt{remark}{theorem}
\newtheorem{remark}[remark]{Remark}
\aliascntresetthe{remark}

\crefname{theorem}{Theorem}{theorems}
\Crefname{theorem}{Theorem}{Theorems}
\crefname{proposition}{Proposition}{propositions}
\Crefname{proposition}{Proposition}{Propositions}
\crefname{lemma}{Lemma}{lemmas}
\Crefname{lemma}{Lemma}{Lemmas}
\crefname{corollary}{Corollary}{corollaries}
\Crefname{corollary}{Corollary}{Corollaries}
\crefname{definition}{Definition}{definitions}
\Crefname{definition}{Definition}{Definitions}
\crefname{remark}{Remark}{remarks}
\Crefname{remark}{Remark}{Remarks}

\newcommand{\dd}{\,d}
\newcommand{\dv}{\,dv}
\newcommand{\ds}{\,ds}
\newcommand{\ii}{\mathrm i}
\newcommand{\e}{\mathrm e}
\newcommand{\tr}{\operatorname{tr}}
\newcommand{\Span}{\operatorname{span}}
\newcommand{\dist}{\operatorname{dist}}
\newcommand{\Id}{\mathrm I}
\newcommand{\Sph}{\mathbb S^2}
\newcommand{\SphK}{\mathbb S_K^2}
\newcommand{\D}{\mathbb D}
\newcommand{\Mmat}{\mathbf M}
\newcommand{\Kmat}{\mathbf K}
\newcommand{\qform}{\mathfrak q}
\newcommand{\cM}{\mathcal M}
\newcommand{\cF}{\mathcal F}

\title[Sharp harmonic-mean inequalities]
{Sharp harmonic-mean inequalities for Neumann and Aharonov--Bohm spectra}

\author{Daguang Chen}
\email{dgchen@tsinghua.edu.cn}
\author{Chengxi Yang}
\email{ycx24@mails.tsinghua.edu.cn}
\address{Department of Mathematical Sciences, Tsinghua University, Beijing 100084, P. R. China.}

\subjclass[2020]{Primary 35P15; Secondary 49R05, 58J50, 81Q10}
\keywords{Neumann eigenvalues, Aharonov--Bohm potentials, Reciprocal sums, Isoperimetric inequalities, Surfaces}
\thanks{The authors were supported by NSFC grant No. 11831005 and NSFC-FWO W2521103.}
\date{\today}

\begin{document}

\begin{abstract}
We prove sharp two-eigenvalue isoperimetric inequalities for  Neumann and Aharonov--Bohm Neumann spectra on surfaces.  If $\Omega\subset\mathbb S^2$ is smooth, simply connected and proper, then the harmonic mean of $\mu_2(\Omega)$ and $\mu_3(\Omega)$ is bounded above by the first positive Neumann eigenvalue of the equal-area geodesic disk, with equality only for disks.  For simply connected surfaces with Gaussian curvature bounded above, we obtain the magnetic analogue for the first two Aharonov--Bohm eigenvalues.  The proof combines a two-dimensional reciprocal Rayleigh--Ritz principle with Green-level comparison.  A key additional ingredient is a spectral ordering theorem for magnetic spherical caps: for $0<\nu<1/2$, the first two eigenvalues lie in the angular sectors of effective orders $\nu$ and $1-\nu$.  We prove this by the factorization $L_0=T^*T$, $L_1=TT^*$ and an exact Neumann--Dirichlet spectral shift.  We also obtain sharp full-sphere and closed-surface bounds, and an annular inequality in terms of conformal modulus and flux.
\end{abstract}

\maketitle

\section{Introduction}\label{sec:introduction}

Sharp isoperimetric inequalities for Neumann eigenvalues begin with the theorem of
Szeg\H{o} \cite{Szego1954}, who showed that the disk maximizes the first positive
Neumann eigenvalue among simply connected planar domains of fixed area.  His
argument also yields the stronger reciprocal inequality
\[
 \frac1{\mu_2(\Omega)}+\frac1{\mu_3(\Omega)}
 \geq \frac2{\mu_2(D)},
\]
where $D$ is the equal-area disk.  Weinberger \cite{Weinberger1956} extended the
first-eigenvalue result to Euclidean domains in arbitrary dimension.  On surfaces
of constant or bounded curvature, corresponding estimates were obtained by
Bandle \cite{Bandle1972,Bandle1980}, Ashbaugh--Benguria
\cite{AshbaughBenguria1995}, and Langford--Laugesen
\cite{LangfordLaugesen2023}.  Recently, Provenzano--Savo
\cite{ProvenzanoSavo2026} proved that every smooth simply connected proper
domain of the round sphere satisfies the sharp comparison with the equal-area
geodesic disk, without an area restriction.

The purpose of this paper is to recover the two-eigenvalue strength of the
planar Szeg\H{o} theorem in the spherical setting, and to establish a magnetic
counterpart for Aharonov--Bohm operators.  The two problems are governed by the
same variational principle but require different geometric trial spaces.

Let
\[
 0=\mu_1(M)<\mu_2(M)\leq\mu_3(M)\leq\cdots
\]
be the Neumann spectrum of a connected compact Riemannian surface.  For a real
magnetic potential $A$, smooth outside finitely many poles, write
\[
 d^A u=du-\ii uA,\qquad \Delta_A=(d^A)^*d^A,
\]
and denote the magnetic Neumann eigenvalues by
\[
 0\leq\lambda_1(M,A)\leq\lambda_2(M,A)\leq\cdots.
\]
For an Aharonov--Bohm potential the spectrum depends only on the flux modulo
integers and complex conjugation, so a nonintegral flux may be reduced to
\[
 0<\nu\leq\frac12.
\]
We shall use the harmonic means
\[
 H_{2,3}(M)=
 \frac{2}{\mu_2(M)^{-1}+\mu_3(M)^{-1}},
 \qquad
 H_{1,2}(M,A)=
 \frac{2}{\lambda_1(M,A)^{-1}+\lambda_2(M,A)^{-1}}.
\]

\subsection{Main results}

Our first result is the spherical reciprocal analogue of the theorem of
Provenzano--Savo.

\begin{theorem}\label{thm:A}
Let $\Omega\subset\Sph$ be a smooth simply connected proper domain and let
$\Omega^\star\subset\Sph$ be a geodesic disk with
$|\Omega^\star|=|\Omega|$. Then
\begin{equation}\label{eq:main-ordinary-intro}
 H_{2,3}(\Omega)\leq H_{2,3}(\Omega^\star)=\mu_2(\Omega^\star).
\end{equation}
Equality holds if and only if $\Omega$ is a geodesic disk.
\end{theorem}

Since $\mu_2\leq\mu_3$, \cref{thm:A} contains the first-eigenvalue comparison of
\cite[Theorem~1.1]{ProvenzanoSavo2026}.  The simple-connectivity hypothesis is
essential: the spherical annuli constructed in
\cite[Theorem~1.2]{ProvenzanoSavo2026} already violate the first-eigenvalue
comparison and hence also \eqref{eq:main-ordinary-intro}.  A curvature-bounded
extension is recorded in \cref{cor:ordinary-curvature}.

We next turn to the magnetic problem.  On a simply connected surface with
boundary, $A_p^{(\nu)}$ denotes a one-pole Aharonov--Bohm potential of flux
$\nu$ at $p$.  On a closed simply connected surface,
$A_{z_1,z_2}^{(\nu)}$ denotes the corresponding two-pole potential, with
opposite fluxes at the two poles.  For $K>0$, let $\SphK$ be the round sphere of
Gaussian curvature $K$, let $\cM_K$ be the class of compact simply connected
surfaces with nonempty smooth boundary and Gaussian curvature at most $K$, and
let $\overline{\cM}_K$ be the analogous closed class.

Michetti--Provenzano--Savo \cite{MichettiProvenzanoSavo2026} proved sharp
first-eigenvalue comparisons for these operators by reducing the problem to
one-dimensional Green-level quotients.  Our main magnetic theorem upgrades that
comparison to the first two eigenvalues.

\begin{theorem}\label{thm:B}
Let $K>0$, let $\Omega\in\cM_K$ satisfy $|\Omega|<4\pi/K$, and fix
$p\in\Omega$ and $0<\nu\leq1/2$.  Let
$\Omega^\star\subset\SphK$ be the geodesic disk of area $|\Omega|$ centered at
$p^\star$. Then
\begin{equation}\label{eq:cap-literal-intro}
 H_{1,2}(\Omega,A_p^{(\nu)})
 \leq
 H_{1,2}(\Omega^\star,A_{p^\star}^{(\nu)}).
\end{equation}
Equality holds if and only if $(\Omega,p)$ is isometric to
$(\Omega^\star,p^\star)$.
\end{theorem}

The passage from a two-sector variational estimate to the literal pair
$\lambda_1,\lambda_2$ in \cref{thm:B} requires a spectral fact that does not
follow from the ground-state comparison.  Let $\Omega_R^\star$ be the spherical
cap of radius $R$ and let $\eta_{s,j}^N(R)$ denote the $j$th radial Neumann
eigenvalue of angular order $s$.  We prove the following ordering theorem.

\begin{theorem}[Magnetic cap ordering]\label{thm:cap-angular-ordering}
Let $0<R<\pi$.
\begin{enumerate}[label=\textup{(\roman*)},leftmargin=2.3em]
\item If $0<\nu<1/2$, then
\[
 \lambda_1(\Omega_R^\star,A_{p^\star}^{(\nu)})
 =\eta_{\nu,0}^N(R),
 \qquad
 \lambda_2(\Omega_R^\star,A_{p^\star}^{(\nu)})
 =\eta_{1-\nu,0}^N(R),
\]
and both eigenvalues are simple over $\mathbb C$.
\item If $\nu=1/2$, then
\[
 \lambda_1(\Omega_R^\star,A_{p^\star}^{(1/2)})
 =
 \lambda_2(\Omega_R^\star,A_{p^\star}^{(1/2)})
 =
 \eta_{1/2,0}^N(R),
\]
and the first eigenspace has complex dimension two.
\end{enumerate}
\end{theorem}

The point is to exclude the first radial excitation in the $n=0$ sector from
crossing the $n=1$ ground state.  If $L_s$ denotes the radial operator of angular
order $s$, then in $L^2((0,R),\sin r\,dr)$
\[
 L_0=T^*T,\qquad L_1=TT^*,
 \qquad
 T=\frac d{dr},\quad T^*=-\frac d{dr}-\cot r.
\]
Hence the positive radial Neumann spectrum of $L_0$ agrees with the radial
Dirichlet spectrum of $L_1$:
\[
 \eta_{0,j+1}^N(R)=\eta_{1,j}^D(R).
\]
Together with strict Dirichlet--Neumann separation and monotonicity in the
angular parameter, this gives
\[
 \eta_{\nu,1}^N(R)
 \geq\eta_{0,1}^N(R)
 =\eta_{1,0}^D(R)
 >\eta_{1,0}^N(R)
 >\eta_{1-\nu,0}^N(R),
\]
which is the required no-crossing inequality.

At and above the full-sphere threshold, the same two-sector principle yields a
universal model bound.

\begin{theorem}\label{thm:C}
Let $K>0$ and $0<\nu\leq1/2$.
\begin{enumerate}[label=\textup{(\roman*)},leftmargin=2.3em]
\item If $\Omega\in\cM_K$, $|\Omega|\geq4\pi/K$, and $p\in\Omega$, then
\begin{equation}\label{eq:large-intro}
 H_{1,2}(\Omega,A_p^{(\nu)})
 \leq
 \frac{K\nu(1-\nu)(\nu+1)(2-\nu)}{\nu^2-\nu+1}.
\end{equation}
The constant is sharp.
\item If $\Sigma\in\overline{\cM}_K$ and $z_1,z_2\in\Sigma$ are distinct, then
\begin{equation}\label{eq:closed-intro}
 H_{1,2}(\Sigma,A_{z_1,z_2}^{(\nu)})
 \leq
 \frac{K\nu(1-\nu)(\nu+1)(2-\nu)}{\nu^2-\nu+1}.
\end{equation}
Equality holds if and only if $\Sigma$ is isometric to $\SphK$ and
$z_1,z_2$ are antipodal.
\end{enumerate}
\end{theorem}

\begin{remark}\label{rem:punctured-model}
In \cref{thm:C}\textup{(i)} equality is not attained in the class with nonempty
smooth boundary.  The sharp constant is approached by centered caps exhausting
the sphere; the limiting finite-energy model is the antipodal two-pole sphere.
\end{remark}

Finally, for a Riemannian annulus $C$, let $M>0$ denote its conformal modulus,
so that $C$ is conformally equivalent to
$C_M=\mathbb S^1\times[-M,M]$ with $\mathbb S^1$ of length $2\pi$.

\begin{theorem}\label{thm:D}
Let $C$ be a Riemannian annulus of conformal modulus $M>0$, and let
$0<\nu\leq1/2$. Then
\begin{equation}\label{eq:annulus-two-mode-intro}
 H_{1,2}(C,A^{(\nu)})
 \leq
 \frac{8\pi M}{|C|}
 \frac{\nu^2(1-\nu)^2}{\nu^2+(1-\nu)^2}.
\end{equation}
Let $\widehat C_M$ be the homothetic copy of $C_M$ with area $|C|$.  If
\begin{equation}\label{eq:cylinder-threshold-intro}
 1-2\nu\leq\frac{\pi^2}{4M^2},
\end{equation}
then
\begin{equation}\label{eq:annulus-model-intro}
 H_{1,2}(C,A^{(\nu)})
 \leq H_{1,2}(\widehat C_M,A^{(\nu)}),
\end{equation}
with equality if and only if $C$ is isometric to $\widehat C_M$.
\end{theorem}

\subsection{Method and relation to previous work}

The proofs are organized around a common two-dimensional variational mechanism.
For two trial functions, a trace form of the Rayleigh--Ritz principle converts
a diagonal energy matrix into a lower bound for the sum of reciprocal
eigenvalues.  In the ordinary problem, the trial pair is obtained from the real
and imaginary parts of a centered integer-flux gauge mode.  In the magnetic
problem, the pair comes from the Green-level radial minimizers with effective
orders $\nu$ and $1-\nu$.  The geometric comparison of
\cite{ProvenzanoSavo2026,MichettiProvenzanoSavo2026} then controls the two
diagonal energies simultaneously.  Thus the genuinely new step is not another
first-eigenvalue comparison, but the construction and identification of a
two-dimensional trial space that is sharp at the model.

Variational principles for reciprocal eigenvalue sums go back to Hersch
\cite{Hersch1961} and Hile--Xu \cite{HileXu1993}.  Recent sharp results for
Neumann reciprocal sums include work of Xia--Wang \cite{WangXia2023},
Benguria--Brandolini--Chiacchio \cite{BenguriaBrandoliniChiacchio2020},
He--Li--Tang \cite{HeLiTang2026}, You--Zhang \cite{YouZhang2026}, and
Eddaoudi \cite{Eddaoudi2026}.  For Aharonov--Bohm eigenvalues we use the
geometric framework developed by Colbois--Provenzano--Savo
\cite{ColboisProvenzanoSavo2022} and Michetti--Provenzano--Savo
\cite{MichettiProvenzanoSavo2026}.  The magnetic annular comparison is related
to the work of Provenzano--Savo \cite{ProvenzanoSavo2025}.

The paper is organized as follows.  \Cref{sec:Magnetic spectral preliminaries and Green-level reductions}
collects the magnetic preliminaries, Green-level reductions, and reciprocal
Rayleigh--Ritz principle.  \Cref{sec:ordinary} proves \cref{thm:A}.
\Cref{sec:magnetic} proves \cref{thm:B,thm:C,thm:D}.  The final section proves
\cref{thm:cap-angular-ordering} and the sharpness statement in
\cref{rem:punctured-model}.

\section{Magnetic spectral preliminaries and Green-level reductions}\label{sec:Magnetic spectral preliminaries and Green-level reductions}
In this section, we collect the magnetic spectral definitions, Green-level reductions, and variational comparisons used in the proofs. The presentation draws mainly on \cite{ProvenzanoSavo2026,MichettiProvenzanoSavo2026} for the magnetic and Green-level framework, and on \cite{Hersch1961,HileXu1993,HeLiTang2026} for the reciprocal variational principle.

\subsection{The magnetic Laplacian and its eigenvalues}
Let $(M,g)$ be a compact oriented Riemannian surface, possibly with boundary, and let $A$ be a smooth real one-form outside a finite set of poles.
With the notation of \cref{sec:introduction}, the magnetic differential, its adjoint, and the magnetic Laplacian are
\begin{equation*}
    \begin{aligned}
        & d^Au=du-\ii uA, \qquad (d^A)^*w=d^*w+\mathrm{i}\langle A,w\rangle_g, \\
        & \Delta_A u=(d^A)^*d^A u = \Delta u+2\mathrm{i}\langle A,du\rangle_g+\left(|A|_g^2-\mathrm{i}d^*A\right)u.
    \end{aligned}
\end{equation*}
If $\partial M\neq\varnothing$, we impose the magnetic Neumann condition
\begin{equation*}
    d^Au(N)=0\qquad\text{on }\partial M,
\end{equation*}
where $N$ is the outward unit normal.
The corresponding quadratic form is
\begin{equation*}
    \qform_A(u,v)=\int_M\langle d^Au,d^Av\rangle_{\mathbb C}\dv,
\end{equation*}
with form domain
\begin{equation}\label{eq:magnetic-form-domain}
    H_A^1(M)=\left\{u\in L^2(M;\mathbb C):d^Au\in L^2(T^*M;\mathbb C)\right\}.
\end{equation}
The differential is understood weakly away from the poles, and \eqref{eq:magnetic-form-domain} specifies the finite-energy realization.
The quadratic form defines a nonnegative self-adjoint operator with compact resolvent, whose eigenvalues satisfy
\begin{equation*}
    0\leq\lambda_1(M,A)\leq\lambda_2(M,A)\leq\cdots\nearrow+\infty.
\end{equation*}
The min--max principle gives
\begin{equation*}
    \lambda_k(M,A)=\min_{\substack{V\subset H_A^1(M)\\ \dim V=k}}\max_{0\neq u\in V}\frac{\int_M|d^Au|^2\dv}{\int_M|u|^2\dv}.
\end{equation*}
When $A=0$, we write $\mu_k(M)$ for the ordinary Neumann eigenvalues and, when $M$ is connected, enumerate them by
\begin{equation*}
    0=\mu_1(M)<\mu_2(M)\leq\mu_3(M)\leq\cdots.
\end{equation*}

\subsection{Aharonov--Bohm potentials and gauge invariance}
Recall that the normalized flux of a closed one-form $A$ around an oriented closed curve $c$ is
\begin{equation*}
    \Phi^A(c)=\frac1{2\pi}\int_cA.
\end{equation*}

Let $\Omega$ be simply connected with nonempty boundary and let $p\in\Omega$.
A one-pole Aharonov--Bohm potential of flux $\nu$ is a smooth closed one-form $A_p^{(\nu)}$ on $\Omega\setminus\{p\}$ such that
\begin{equation*}
    \Phi^{A_p^{(\nu)}}(c)=\nu
\end{equation*}
for every positively oriented simple loop enclosing $p$.
Any two such potentials differ by an exact form and hence define isospectral operators.

Let now $\Sigma$ be a closed simply connected surface and let $z_1,z_2\in\Sigma$ be distinct.
A two-pole Aharonov--Bohm potential of flux $\nu$ is a smooth closed one-form $A_{z_1,z_2}^{(\nu)}$ on $\Sigma\setminus\{z_1,z_2\}$ whose fluxes around $z_1$ and $z_2$ are $\nu$ and $-\nu$, respectively.
For a fixed metric, its spectrum depends only on the flux and the positions of the poles.

The following gauge properties are recorded in \cite[Lemma~2.3]{ProvenzanoSavo2026} and \cite[Theorem~1.1]{MichettiProvenzanoSavo2026}. If $A-A'$ is closed and every normalized period $\frac1{2\pi}\int_c(A-A')$ is an integer, then
\begin{equation*}
    \lambda_k(M,A)=\lambda_k(M,A')\qquad\text{for every }k\geq1.
\end{equation*}
Moreover, $\lambda_1(M,A)=0$ if and only if $A$ is closed and all its normalized periods are integers.

For an Aharonov--Bohm potential of flux $\gamma\in\mathbb R$, gauge invariance and complex conjugation show that its spectrum depends only on
\begin{equation*}
    \nu=\dist(\gamma,\mathbb Z)\in[0,1/2].
\end{equation*}
Thus the nonintegral-flux estimates reduce to $0<\nu\leq1/2$. We also record the integer-flux gauge transformation.
\begin{lemma}\label{lem:gauge}
    Let $A$ be a closed one-form whose normalized periods are integers. Then there exists a smooth $\mathbb{S}^{1}$-valued function $F$ on the punctured surface such that
    \begin{equation*}
        dF=\ii FA, \qquad d^A(Fu)=F\,du, \qquad (d^A)^*(F\omega)=F\,d^*\omega, \qquad \Delta_A=F\circ\Delta\circ F^{-1}.
    \end{equation*}
    Consequently, multiplication by $F$ gives a unitary equivalence between the ordinary and magnetic Neumann problems, and hence $\lambda_k(M,A)=\mu_k(M)$ for every $k\geq1$.

    Conversely, if there exists a nonzero function $u$ satisfying $d^Au=0$, then all normalized periods of $A$ are integers. In particular, if $A$ has a nonintegral normalized period, then $\ker\Delta_A=\{0\}$ and $\lambda_1(M,A)>0$.
\end{lemma}
\begin{proof}
    Fix $x_0$ and define $F(x)=\exp\bigl(\ii\int_{c_x}A\bigr)$, where $c_x$ joins $x_0$ to $x$. The integer-period condition makes $F$ independent of the path, and differentiation gives the stated identities.

    Conversely, suppose that $d^Au=0$ for some nonzero $u$. Then $d|u|^2=0$, so $|u|$ is a positive constant. Thus $F_u=u/|u|$ is an $\mathbb{S}^{1}$-valued function satisfying $dF_u=\ii F_uA$. For every closed curve $c$,
    \begin{equation*}
        \frac1{2\pi}\int_cA=\frac1{2\pi\ii}\int_cF_u^{-1}dF_u=\deg(F_u|_c)\in\mathbb Z.
    \end{equation*}
    Hence all normalized periods of $A$ are integers. Finally, any function in $\ker\Delta_A$ satisfies $\|d^Au\|_2^2=0$, which proves the last assertion.
\end{proof}

\subsection{Green functions and canonical potentials}
Let $\Omega$ be a compact simply connected surface with nonempty smooth boundary and let $p\in\Omega$. The Dirichlet Green function $\psi_p$ is determined by
\begin{equation*}
    \Delta\psi_p=\delta_p\quad\text{in }\Omega, \qquad \psi_p=0\quad\text{on }\partial\Omega.
\end{equation*}
The function $\psi_p$ is positive and harmonic in $\Omega\setminus\{p\}$. Let $\D\subset\mathbb R^2$ be the unit disk and let $(r,\theta)$ be the usual polar coordinates centered at the origin. If $\Phi:\D\to\Omega$ is a conformal diffeomorphism with $\Phi(0)=p$, then
\begin{equation*}
    \Phi^*\psi_p=-\frac1{2\pi}\log r.
\end{equation*}
Consequently, $\psi_p$ has no critical points in $\Omega\setminus\{p\}$, and its level lines are smooth simple closed curves.

The canonical unit-flux potential and its flux-$\nu$ multiple are
\begin{equation}\label{eq:unit-flux-potential}
    A_p=-2\pi*d\psi_p, \qquad A_p^{(\nu)}=\nu A_p=-2\pi\nu*d\psi_p,
\end{equation}
where $*$ is the Hodge-star operator associated with $g$. The form $A_p$ is smooth, closed, and co-closed away from $p$, has flux one around $p$, and satisfies $\Phi^*A_p=d\theta$.

Let now $\Sigma$ be a closed simply connected surface and let $z_1,z_2\in\Sigma$ be distinct. We denote by $\psi_{z_1,z_2}$ the two-pole Green function normalized by
\begin{equation*}
    \Delta\psi_{z_1,z_2}=\delta_{z_1}-\delta_{z_2},
    \qquad
    \int_\Sigma\psi_{z_1,z_2}\dv=0.
\end{equation*}
The associated canonical potential is
\begin{equation*}
    A_{z_1,z_2} = -2\pi*d\psi_{z_1,z_2}, \qquad A_{z_1,z_2}^{(\nu)} = \nu A_{z_1,z_2} = -2\pi\nu*d\psi_{z_1,z_2}.
\end{equation*}
The flux-$\nu$ potential has fluxes $\nu$ and $-\nu$ around $z_1$ and $z_2$, respectively. After choosing a conformal parametrization $\Phi:\overline{\mathbb C}\to\Sigma$ with $\Phi(0)=z_1$ and $\Phi(\infty)=z_2$, one has $\Phi^*\psi_{z_1,z_2}=-\frac{1}{2\pi}\log r+C$ and $\Phi^*A_{z_1,z_2}^{(\nu)}=\nu\,d\theta$. Thus $\psi_{z_1,z_2}$ has no critical points away from its poles, assumes all real values, and has smooth simple closed level lines.

\subsection{Green-level coordinates and radial spectra}
We first consider the one-pole case and set $M=|\Omega|$. For $t>0$, let
\begin{equation*}
    \alpha_p(t)=|\{\psi_p>t\}|.
\end{equation*}
The coarea formula gives $\alpha_p'(t)=-\int_{\{\psi_p=t\}}|d\psi_p|^{-1}\ds$.
Since $\psi_p$ has no critical points away from $p$, the function $\alpha_p:(0,\infty)\to(0,M)$ is smooth and strictly decreasing. Let $\beta_p$ be its inverse and define
\begin{equation*}
    G_p(a)=\int_{\{\psi_p=\beta_p(a)\}}\frac1{|d\psi_p|}\ds, \qquad 0<a<M.
\end{equation*}
For $a=\alpha_p(t)$ and $t=\beta_p(a)$, we have
\begin{equation*}
    \alpha_p'(t)=-G_p(a), \qquad \beta_p'(a)=-\frac{1}{G_p(a)}, \qquad da=-G_p(a)\,dt=-G_p(a)\,d\psi_p.
\end{equation*}
Moreover,
\begin{equation}\label{eq:G-one-asymptotic}
    G_p(a) \sim 4\pi a \quad \text{as } a\to0,
\end{equation}
while $a=M$ is a regular endpoint. We denote the corresponding Green-level class by
\begin{equation*}
    \mathcal R_p(\Omega)=\{g\circ\psi_p:g\in H^1(0,\infty)\}\subset H_{A_p^{(\nu)}}^1(\Omega),
\end{equation*}
and define its first radial eigenvalue by
\begin{equation*}
    \kappa_1(\Omega,A_p^{(\nu)})=\inf_{0\neq u\in\mathcal R_p(\Omega)}\frac{\int_\Omega|d^{A_p^{(\nu)}}u|^2\dv}{\int_\Omega u^2\dv}.
\end{equation*}

The two-pole case is analogous. Set $M=|\Sigma|$ and define
\begin{equation*}
    \alpha_{z_1,z_2}(t)=|\{\psi_{z_1,z_2}>t\}|, \qquad t\in\mathbb R.
\end{equation*}
The function $\alpha_{z_1,z_2}:\mathbb R\to(0,M)$ is smooth and strictly decreasing. Let $\beta_{z_1,z_2}$ be its inverse and define
\begin{equation*}
    G_{z_1,z_2}(a)=\int_{\{\psi_{z_1,z_2}=\beta_{z_1,z_2}(a)\}}\frac1{|d\psi_{z_1,z_2}|}\ds, \qquad 0<a<M.
\end{equation*}
Its endpoint behavior is
\begin{equation}\label{eq:G-two-asymptotic}
    G_{z_1,z_2}(a)\sim4\pi a \quad\text{as }a\to0, \qquad G_{z_1,z_2}(a)\sim4\pi(M-a) \quad\text{as }a\to M.
\end{equation}
We denote the corresponding Green-level class by
\begin{equation*}
    \mathcal R_{z_1,z_2}(\Sigma) =\{g\circ\psi_{z_1,z_2}:g\in H^1(\mathbb R)\} \subset H_{A_{z_1,z_2}^{(\nu)}}^1(\Sigma),
\end{equation*}
and define its first radial eigenvalue by
\begin{equation*}
    w_1(\Sigma,A_{z_1,z_2}^{(\nu)})=\inf_{0\neq u\in\mathcal R_{z_1,z_2}(\Sigma)}\frac{\int_\Sigma|d^{A_{z_1,z_2}^{(\nu)}}u|^2\dv}{\int_\Sigma u^2\dv}.
\end{equation*}

In either case, write $X$ for the surface and $\psi$, $\beta$, $G$, and $A^{(\nu)}$ for the Green function, inverse area function, level coefficient, and canonical potential, respectively. Let $u=g\circ\psi$ be real-valued and set $f=g\circ\beta$. Then
\begin{equation*}
    g(t)=f(a), \qquad g'(t)=-G(a)f'(a).
\end{equation*}
Since $A^{(\nu)}=-2\pi\nu*d\psi$ and $d\psi$ is orthogonal to $*d\psi$,
\begin{equation*}
    |d^{A^{(\nu)}}u|^2=\big(g'(t)^2+4\pi^2\nu^2g(t)^2\big)|d\psi|^2.
\end{equation*}
Using $\int_{\{\psi=t\}}|d\psi|\ds=1$, the coarea formula yields
\begin{equation*}
    \begin{aligned}
        \int_X u^2\dv=\int_0^M f^2\dd a \qquad \text{and} \qquad
        \int_X|d^{A^{(\nu)}}u|^2\dv=\int_0^M\left(Gf'^2+\frac{4\pi^2\nu^2}{G}f^2\right)\dd a.
    \end{aligned}
\end{equation*}

These identities yield the following one-dimensional characterizations; see \cite[Lemma~3.1]{ProvenzanoSavo2026} and \cite[Lemmas~3.1 and~4.1]{MichettiProvenzanoSavo2026}.
\begin{proposition}
    Let $\nu>0$. For a positive level coefficient $G$ on $(0,M)$, set
    \begin{equation*}
        \cF_G=\left\{ f\in L^2(0,M)\cap H^1_{\mathrm{loc}}(0,M): \sqrt Gf'\in L^2(0,M),\ G^{-1/2}f\in L^2(0,M) \right\}.
    \end{equation*}
    \begin{enumerate}[label=\textup{(\roman*)},leftmargin=2.3em]
        \item In the one-pole case,
        \begin{equation*}
            \begin{aligned}
                \kappa_1(\Omega,A_p^{(\nu)})=\kappa_1(G_p;\nu):=\inf_{0\neq f\in\cF_{G_p}}\frac{\int_0^M\left(G_pf'^2+\frac{4\pi^2\nu^2}{G_p}f^2\right)\dd a}{\int_0^M f^2\dd a}.
            \end{aligned}
        \end{equation*}
        Equivalently, it is the first eigenvalue $\kappa=\kappa_1(G_p;\nu)$ of the Sturm--Liouville problem
        \begin{equation*}
            \begin{cases}
                & -(G_pf')'+\frac{4\pi^2\nu^2}{G_p}f=\kappa f, \qquad \text{in }(0,M),\\
                & \displaystyle\lim_{a\to0}G_p(a)f'(a)=0, \qquad f'(M)=0.
            \end{cases}
        \end{equation*}
        In particular, $\lambda_1(\Omega,A_p^{(\nu)})\leq\kappa_1(\Omega,A_p^{(\nu)})$.

        \item In the two-pole case,
        \begin{equation*}
            \begin{aligned}
                w_1(\Sigma,A_{z_1,z_2}^{(\nu)})=w_1(G_{z_1,z_2};\nu):=\inf_{0\neq f\in\cF_{G_{z_1,z_2}}}\frac{\int_0^M\left(G_{z_1,z_2}f'^2+\frac{4\pi^2\nu^2}{G_{z_1,z_2}}f^2\right)\dd a}{\int_0^M f^2\dd a}.
            \end{aligned}
        \end{equation*}
        Equivalently, it is the first eigenvalue $w=w_1(G_{z_1,z_2};\nu)$ of the Sturm--Liouville problem
        \begin{equation*}
            \begin{cases}
                & -(G_{z_1,z_2}f')'+\frac{4\pi^2\nu^2}{G_{z_1,z_2}}f=wf, \qquad \text{in }(0,M),\\
                & \displaystyle\lim_{a\to0}G_{z_1,z_2}(a)f'(a)=\lim_{a\to M}G_{z_1,z_2}(a)f'(a)=0.
            \end{cases}
        \end{equation*}
        In particular, $\lambda_1(\Sigma,A_{z_1,z_2}^{(\nu)})\leq w_1(\Sigma,A_{z_1,z_2}^{(\nu)})$.
    \end{enumerate}

    In both cases, the first eigenvalue is simple and has a positive $L^2(0,M)$-normalized minimizer.
\end{proposition}

\subsection{Monotonicity and spherical comparison}
The following comparisons are established in \cite[Lemmas~3.2 and~5.1]{ProvenzanoSavo2026} and \cite[Lemmas~3.2, 3.3 and~4.3]{MichettiProvenzanoSavo2026}. Although the latter paper states monotonicity for a reduced flux, the argument requires only $\nu>0$. It therefore also applies to $1-\nu$ when $\nu\in(0,\frac{1}{2}]$.
\begin{lemma}\label{lem:monotonicity}
    Let $G_1$ and $G_2$ be smooth positive functions with $0<G_1\leq G_2$ on $(0,M)$, and let $\nu>0$.

    \begin{enumerate}[label=\textup{(\roman*)},leftmargin=2.3em]
        \item In the one-pole case, if $G_1,G_2$ satisfy \eqref{eq:G-one-asymptotic}, then
        \begin{equation*}
            \kappa_1(G_2;\nu)\leq\kappa_1(G_1;\nu).
        \end{equation*}

        \item In the two-pole case, if $G_1,G_2$ satisfy \eqref{eq:G-two-asymptotic}, then
        \begin{equation*}
            w_1(G_2;\nu)\leq w_1(G_1;\nu).
        \end{equation*}
    \end{enumerate}

    Each inequality is strict if $G_1<G_2$ on a set of positive measure.
\end{lemma}

Let $\Omega^\star\subset\Sph$ be a geodesic disk of radius $R$ and center $p^*$, and set $M=|\Omega^\star|$. In geodesic polar coordinates $(r,\theta)$ centered at $p^*$, its Green function is
\begin{equation*}
    \psi_{p^*}(r)=-\frac1{2\pi}\log\left(\frac{\tan(r/2)}{\tan(R/2)}\right).
\end{equation*}
The corresponding potential and level coefficient are
\begin{equation*}
    A_{p^*}^{(\nu)}=-2\pi\nu*d\psi_{p^*}=\nu\,d\theta
\end{equation*}
and
\begin{equation*}
    G^\star(a)=G_{p^*}(a)=a(4\pi-a), \qquad 0<a<M.
\end{equation*}
These explicit computations are given in \cite[Section~3.2]{MichettiProvenzanoSavo2026}. 
Applying the sharp isoperimetric inequality \cite[Isoperimetric Inequality~(I)]{ChavelFeldman1980} to each superlevel set of $\psi_p$ gives the following comparison.
\begin{lemma}\label{lem:spherical-level-comparison}
    Let $\Omega$ be a compact simply connected surface with smooth boundary, Gaussian curvature at most one and $|\Omega|=M<4\pi$. Then, for every $p\in\Omega$,
    \begin{equation*}
        G_p(a)\geq G^\star(a), \qquad 0<a<M.
    \end{equation*}
    Equality holds almost everywhere if and only if $(\Omega,p)$ is isometric to $(\Omega^\star,p^*)$.
\end{lemma}

Combining \cref{lem:monotonicity,lem:spherical-level-comparison}, we obtain, for every $\nu>0$,
\begin{equation}\label{eq:kappa-spherical-comparison}
    \kappa_1(\Omega,A_p^{(\nu)}) \leq \kappa_1(\Omega^\star,A_{p^*}^{(\nu)}).
\end{equation}
By the strict part of \cref{lem:monotonicity}, equality holds if and only if $(\Omega,p)$ is isometric to $(\Omega^\star,p^*)$.

\subsection{A two-dimensional reciprocal variational principle}
We use the following two-dimensional matrix form of the reciprocal variational principles of Hersch and of Hile and Xu \cite{Hersch1961,HileXu1993}.
\begin{lemma}\label{lem:trace-ritz}
    Let $A$ be one of the Aharonov--Bohm potentials considered above, with flux $\nu\in[0,1/2]$, and let $\phi_1,\phi_2 \in H_A^1(M)$ be linearly independent trial functions. Define their mass and energy matrices by
    \begin{equation*}
        \Mmat_{ij}=\int_M\phi_i\overline{\phi_j}\dv,
        \qquad
        \Kmat_{ij}=\qform_A(\phi_i,\phi_j)=\int_M\langle d^A\phi_i,d^A\phi_j\rangle_{\mathbb C}\dv.
    \end{equation*}
    Then the following assertions hold.
    \begin{enumerate}[label=\textup{(\roman*)},leftmargin=2.3em]
        \item If $0<\nu\leq1/2$, then
        \begin{equation*}
            \frac1{\lambda_1(M,A)}+\frac1{\lambda_2(M,A)}
            \geq\tr(\Kmat^{-1}\Mmat).
        \end{equation*}

        \item If $\nu=0$, then by gauge invariance we may assume that $A=0$. If
        \begin{equation*}
            \int_M\phi_1\dv=\int_M\phi_2\dv=0,
        \end{equation*}
        then
        \begin{equation*}
            \frac1{\mu_2(M)}+\frac1{\mu_3(M)}
            \geq\tr(\Kmat^{-1}\Mmat).
        \end{equation*}
    \end{enumerate}
    In either case, equality implies that $\Span_{\mathbb C}\{\phi_1,\phi_2\}$ is invariant under the corresponding operator and is generated by eigenfunctions associated with the two eigenvalues on the left.
\end{lemma}
\begin{proof}
    The Rayleigh--Ritz argument in \cite[Lemma~2.1]{HeLiTang2026} applies once $\Kmat$ is known to be positive definite. It suffices to verify this property.

    For $c=(c_1,c_2)^\dagger\in\mathbb C^2$, set $\phi=c_1\phi_1+c_2\phi_2$. Then
    \begin{equation*}
        c^\dagger\Kmat c=\int_M|d^A\phi|^2\dv.
    \end{equation*}
    Suppose that $c^\dagger\Kmat c=0$.
    If $0<\nu\leq1/2$, then $d^A\phi=0$, and \cref{lem:gauge} implies that $\phi=0$ because $A$ has a nonintegral normalized period.
    If $\nu=0$, then $A=0$ and $d\phi=0$, so $\phi$ is constant. Since $\phi_1$ and $\phi_2$ have zero mean, $\phi$ also has zero mean and hence vanishes.
    In either case, the linear independence of $\phi_1$ and $\phi_2$ implies that $c=0$. Therefore, $\Kmat$ is positive definite, and the stated conclusions follow.
\end{proof}
The two Ritz values are the generalized eigenvalues of $\Kmat a=\eta\Mmat a$; their reciprocal sum is $\tr(\Kmat^{-1}\Mmat)$.
In particular, if $\Kmat=\operatorname{diag}(E_1,E_2)$, then
\begin{equation*}
    \tr(\Kmat^{-1}\Mmat)=\frac{\|\phi_1\|_2^2}{E_1}+\frac{\|\phi_2\|_2^2}{E_2}.
\end{equation*}

\section{Proof of \cref{thm:A}}\label{sec:ordinary}
Throughout this section, $\Omega\subset\Sph$ is a smooth simply connected proper domain, and $\Omega^\star\subset\Sph$ is a geodesic disk with $|\Omega^\star|=|\Omega|$. Let $p^*$ be the center of $\Omega^\star$ and set $\nu=1$. For each $p\in\Omega$, let $A_p$ be the unit-flux potential in \eqref{eq:unit-flux-potential}. We use this notation in the proof of \cref{thm:A}.

\subsection{Radial comparison and centering}
For each $p\in\Omega$, let $u_p\in\mathcal R_p(\Omega)$ be the positive Green-level minimizer associated with $\kappa_1(\Omega,A_p)=\kappa_1(G_p;1)$, normalized by
\begin{equation}\label{eq:up-normalization}
    \int_\Omega u_p^2\dv=1.
\end{equation}

We use the following eigenvalue comparison and centering results of Provenzano and Savo \cite[Theorems~4.1--4.3 and Claim~7.1]{ProvenzanoSavo2026}. 
The inequality and the first identity in \textup{(i)} are their Theorems~4.1 and~4.2; the remaining identities follow from gauge invariance and rotational symmetry. The inequality also follows from \eqref{eq:kappa-spherical-comparison} with $\nu=1$. 
Assertion~\textup{(ii)} is the centering property in the proof of their Theorem~4.3, formulated explicitly in Claim~7.1.
\begin{proposition}\label{prop:PS-inputs}
    The following assertions hold.
    \begin{enumerate}[label=\textup{(\roman*)},leftmargin=2.3em]
        \item For every $p\in\Omega$,
        \begin{equation*}
            \kappa_1(\Omega,A_p)\leq\kappa_1(\Omega^\star,A_{p^*})=\lambda_2(\Omega^\star,A_{p^*})=\mu_2(\Omega^\star)=\mu_3(\Omega^\star).
        \end{equation*}
        Equality in the first inequality holds if and only if $(\Omega,p)$ is isometric to $(\Omega^\star,p^*)$.

        \item There exists $\bar p\in\Omega$ such that
        \begin{equation}\label{eq:PS-centering}
            \int_\Omega u_{\bar p}F_{\bar p}\dv=0,
        \end{equation}
        where $F_{\bar p}$ is an $\mathbb{S}^{1}$-valued gauge factor satisfying $dF_{\bar p}=\ii F_{\bar p}A_{\bar p}$.
    \end{enumerate}
\end{proposition}

\subsection{The two real trial functions}
Fix $\bar p$ as in \cref{prop:PS-inputs} and write
\begin{equation*}
    u=u_{\bar p}, \qquad A=A_{\bar p}, \qquad F=F_{\bar p}, \qquad \kappa=\kappa_1(\Omega,A_{\bar p}).
\end{equation*}
Writing $F=\e^{\ii\Theta}$ locally, define
\begin{equation*}
    P_1=u\cos\Theta, \qquad P_2=u\sin\Theta.
\end{equation*}
These functions are globally defined by $P_1-\ii P_2=uF^{-1}$. Moreover, $d(uF^{-1})=F^{-1}d^Au$, and the complex conjugate of \eqref{eq:PS-centering} gives
\begin{equation*}
    \int_\Omega P_1\dv=\int_\Omega P_2\dv=0.
\end{equation*}
The functions are linearly independent because $u>0$ away from $\bar p$ and $F$ has winding number one along each Green level line.

Let $\Mmat$ and $\Kmat$ be the mass and energy matrices of $P_1,P_2$. Since $P_1^2+P_2^2=u^2$, \eqref{eq:up-normalization} gives
\begin{equation}\label{eq:ordinary-mass-trace}
    \tr\Mmat=1.
\end{equation}
Choose a conformal map $\Phi:\D\to\Omega$ with $\Phi(0)=\bar p$. After multiplying $F$ by a constant of modulus one, we may write
\begin{equation*}
    \Phi^*\psi_{\bar p}=-\frac1{2\pi}\log r, \qquad \Phi^*F=\e^{\ii\theta}, \qquad \Phi^*u=v(r).
\end{equation*}
Hence
\begin{equation*}
    \Phi^*P_1=v(r)\cos\theta, \qquad \Phi^*P_2=v(r)\sin\theta.
\end{equation*}
The bilinear Dirichlet energy is conformally invariant in dimension two. Integrating in $\theta$, we obtain
\begin{equation*}
    \Kmat_{11}=\Kmat_{22}=\pi\int_0^1\left(rv'(r)^2+\frac{v(r)^2}{r}\right)\dd r, \qquad \Kmat_{12}=0.
\end{equation*}
Thus $\Kmat=k\Id_2$ for some $k>0$. On the other hand,
\begin{equation*}
    2k=\int_\Omega\left(|dP_1|^2+|dP_2|^2\right)\dv=\int_\Omega|d(uF^{-1})|^2\dv=\int_\Omega|d^Au|^2\dv=\kappa.
\end{equation*}
Consequently,
\begin{equation}\label{eq:ordinary-K-scalar}
    \Kmat=\frac\kappa2\Id_2.
\end{equation}

\begin{proof}[Proof of \cref{thm:A}]
    By \cref{lem:trace-ritz}, \eqref{eq:ordinary-mass-trace}, and \eqref{eq:ordinary-K-scalar},
    \begin{equation}\label{eq:ordinary-trace-bound}
        \frac1{\mu_2(\Omega)}+\frac1{\mu_3(\Omega)}\geq\tr(\Kmat^{-1}\Mmat)=\frac2\kappa\geq\frac2{\mu_2(\Omega^\star)},
    \end{equation}
    where the last inequality follows from \cref{prop:PS-inputs}. This is equivalent to \eqref{eq:main-ordinary-intro}.

    If $\Omega$ is a geodesic disk, choose $\bar p$ to be its center. Then $\kappa=\mu_2(\Omega^\star)$, and the trial functions $P_1,P_2$ span the first positive eigenspace of $\Omega$. Hence equality holds in \eqref{eq:main-ordinary-intro}.
    Conversely, suppose that equality holds in \eqref{eq:main-ordinary-intro}. Then both inequalities in \eqref{eq:ordinary-trace-bound} are equalities, and in particular $\kappa=\mu_2(\Omega^\star)$. The rigidity statement in \cref{prop:PS-inputs}\textup{(i)} implies that $(\Omega,\bar p)$ is isometric to $(\Omega^\star,p^*)$. Thus $\Omega$ is a geodesic disk.
\end{proof}

The same argument gives the following extension to compact simply connected surfaces with an upper bound on Gaussian curvature. At $K>0$ and $|\Omega|_g=4\pi/K$, the required radial comparison also follows from \cref{lem:full-sphere-sector} with $s=1$.
\begin{corollary}\label{cor:ordinary-curvature}
    Let $(\Omega,g)$ be a compact simply connected Riemannian surface with nonempty smooth boundary and Gaussian curvature bounded above by $K\in\mathbb R$. Assume
    \begin{equation*}
        4\pi-K|\Omega|_g\geq0.
    \end{equation*}
    Let $\Omega_K^\star$ be the constant-curvature-$K$ geodesic disk of area $|\Omega|_g$. When $K>0$ and $|\Omega|_g=4\pi/K$, the model $\Omega_K^\star$ is understood as the full sphere $\SphK$. Then
    \begin{equation*}
        \frac1{\mu_2(\Omega,g)}+\frac1{\mu_3(\Omega,g)}\geq\frac2{\mu_2(\Omega_K^\star)}.
    \end{equation*}
\end{corollary}
\begin{remark}
    Provenzano and Savo established the sharp comparison $\mu_2(\Omega,g)\leq\mu_2(\Omega_K^\star)$ in \cite[Theorem~1.3]{ProvenzanoSavo2026}. Since $\mu_2(\Omega,g)\leq\mu_3(\Omega,g)$, \cref{cor:ordinary-curvature} also implies their inequality.
\end{remark}

\section{Harmonic mean inequalities for Aharonov--Bohm eigenvalues}\label{sec:magnetic}
Throughout this section, we assume $0<\nu\leq1/2$. Recall that, for $K>0$, the classes $\cM_K$ and $\overline{\cM}_K$ are defined in \cref{sec:introduction}. We consider surfaces $\Omega\in\cM_K$ and $\Sigma\in\overline{\cM}_K$, respectively.

\subsection{The two-dimensional trial space}\label{sec:two-sector}
Fix a point $p\in\Omega$ in the one-pole case or distinct points $z_1,z_2\in\Sigma$ in the two-pole case. For $s>0$, use the unified notation
\begin{equation*}
    \big(X,\psi,G,A,\kappa_1(G;s)\big)=
    \begin{cases}
        \big(\Omega,\psi_p,G_p,A_p,\kappa_1(G_p;s)\big), & \text{in the one-pole case},\\
        \big(\Sigma,\psi_{z_1,z_2},G_{z_1,z_2},A_{z_1,z_2},w_1(G_{z_1,z_2};s)\big), & \text{in the two-pole case},
    \end{cases}
\end{equation*}
and set $M=|X|$. In either case, set $A^{(\nu)}=\nu A$. Thus $A^{(\nu)}=A_p^{(\nu)}$ on $\Omega\setminus\{p\}$ and $A^{(\nu)}=A_{z_1,z_2}^{(\nu)}$ on $\Sigma\setminus\{z_1,z_2\}$, respectively.

Since all normalized periods of $A$ are integers, \cref{lem:gauge} provides a smooth $\mathbb{S}^{1}$-valued function $F$ on the corresponding punctured surface such that
\begin{equation*}
    dF=\ii FA.
\end{equation*}
Let $f_0$ and $f_1$ in $\cF_G$ be the positive $L^2(0,M)$-normalized minimizers for $\nu$ and $1-\nu$, respectively. Pulling them back by the area coordinate $a$, set
\begin{equation*}
    \phi_0=f_0(a), \qquad \phi_1=Ff_1(a).
\end{equation*}
A direct computation gives
\begin{equation*}
    d^{A^{(\nu)}}\phi_0=f_0'(a)\,da-\ii\nu Af_0(a), \qquad d^{A^{(\nu)}}\phi_1=F\Big(f_1'(a)\,da+\ii(1-\nu)Af_1(a)\Big).
\end{equation*}
The diagonal energies are
\begin{equation}\label{eq:energy-diagonal-energies}
    \qform_{A^{(\nu)}}(\phi_0,\phi_0)=\kappa_1(G;\nu),\qquad \qform_{A^{(\nu)}}(\phi_1,\phi_1)=\kappa_1(G;1-\nu),
\end{equation}
The functions $\phi_0$ and $\phi_1$ are linearly independent: $\phi_0$ is constant on each regular Green level line, whereas $\phi_1$ has winding number one.

It remains to compute the energy cross term. The identities $da=-G(a)\,d\psi$, $A\perp d\psi$, and $|A|=2\pi|d\psi|$ give
\begin{equation}\label{eq:cross-integrand-form}
    \left\langle d^{A^{(\nu)}}\phi_0,d^{A^{(\nu)}}\phi_1\right\rangle_{\mathbb C}=F^{-1}H(a)|d\psi|^2
\end{equation}
for a real function $H$ depending only on $a$. Along a regular level line, choose the orientation for which $d\Theta=A$ and $F=\e^{\ii\Theta}$. Since $d\Theta/ds=2\pi|d\psi|$,
\begin{equation*}
    \int_{\{\psi=t\}}F^{-1}|d\psi|\ds=\frac1{2\pi}\int_0^{2\pi}\e^{-\ii\Theta}\dd\Theta=0.
\end{equation*}
The coarea formula applied to \eqref{eq:cross-integrand-form} therefore yields
\begin{equation}\label{eq:energy-cross-zero}
    \qform_{A^{(\nu)}}(\phi_0,\phi_1)=0.
\end{equation}
\begin{proposition}\label{prop:two-sector}
    The following reciprocal estimates hold.
    \begin{enumerate}[label=\textup{(\roman*)},leftmargin=2.3em]
        \item In the one-pole case,
        \begin{equation*}
            \frac1{\lambda_1(\Omega,A_p^{(\nu)})}+\frac1{\lambda_2(\Omega,A_p^{(\nu)})}\geq\frac1{\kappa_1(G_p;\nu)}+\frac1{\kappa_1(G_p;1-\nu)}.
        \end{equation*}

        \item In the two-pole case,
        \begin{equation*}
            \frac1{\lambda_1(\Sigma,A_{z_1,z_2}^{(\nu)})}+\frac1{\lambda_2(\Sigma,A_{z_1,z_2}^{(\nu)})}\geq\frac1{w_1(G_{z_1,z_2};\nu)}+\frac1{w_1(G_{z_1,z_2};1-\nu)}.
        \end{equation*}
    \end{enumerate}
\end{proposition}
\begin{proof}
    In the one-pole case, \eqref{eq:energy-diagonal-energies} and \eqref{eq:energy-cross-zero} give
    \begin{equation*}
        \Kmat=\operatorname{diag}\bigl(\kappa_1(G_p;\nu),\kappa_1(G_p;1-\nu)\bigr), \qquad \Mmat_{11}=\Mmat_{22}=1.
    \end{equation*}
    Hence \cref{lem:trace-ritz} implies the conclusion. The two-pole proof is identical, with $w_1(G_{z_1,z_2};\,\cdot\,)$ in place of $\kappa_1(G_p;\,\cdot\,)$.
\end{proof}

\subsection{Spherical caps and spectral ordering}\label{sec:caps}
We first assume $K=1$. Let $\Omega^\star\subset\Sph$ be the geodesic disk with $|\Omega^\star|=|\Omega|<4\pi$ and center $p^*$. By \eqref{eq:kappa-spherical-comparison}, applied to $\nu$ and $1-\nu$, and \cref{prop:two-sector},
\begin{equation}\label{eq:cap-sector-comparison}
    \frac1{\lambda_1(\Omega,A_p^{(\nu)})}+\frac1{\lambda_2(\Omega,A_p^{(\nu)})}\geq\frac1{\kappa_1(\Omega^\star,A_{p^*}^{(\nu)})}+\frac1{\kappa_1(\Omega^\star,A_{p^*}^{(1-\nu)})}.
\end{equation}
The spectral ordering proved in \cref{thm:cap-angular-ordering} identifies the two radial values on the right as
\begin{equation*}
    \lambda_1(\Omega^\star,A_{p^*}^{(\nu)})=\kappa_1(\Omega^\star,A_{p^*}^{(\nu)}), \qquad
    \lambda_2(\Omega^\star,A_{p^*}^{(\nu)})=\kappa_1(\Omega^\star,A_{p^*}^{(1-\nu)}).
\end{equation*}
These identities hold for every cap radius; at $\nu=1/2$, the two eigenvalues coincide. The radial analysis is given in \cref{sec:cap-appendix}.
\begin{proof}[Proof of \cref{thm:B}]
    The preceding identities and \eqref{eq:cap-sector-comparison} prove \eqref{eq:cap-literal-intro} when $K=1$.

    If equality holds, both radial comparisons used in \eqref{eq:cap-sector-comparison} must be equalities. In particular,
    \begin{equation*}
        \kappa_1(\Omega,A_p^{(\nu)})=\kappa_1(\Omega^\star,A_{p^*}^{(\nu)}).
    \end{equation*}
    The strict comparison in \cref{lem:monotonicity}, together with \cref{lem:spherical-level-comparison}, implies that $(\Omega,p)$ is isometric to $(\Omega^\star,p^*)$. Conversely, for a centered geodesic disk the trial pair consists of eigenfunctions associated with its first two magnetic eigenvalues, so equality holds.

    For general $K>0$, replace $g$ by $Kg$. The rescaled metric has Gaussian curvature at most one, its area is $K|\Omega|<4\pi$, and its eigenvalues are $K^{-1}$ times those of $g$. Applying the preceding result and scaling back proves the theorem, with the same equality characterization.
\end{proof}

\subsection{Large-area and closed surfaces}
We use the following radial consequence of the full-sphere comparison argument of Michetti, Provenzano, and Savo; see \cite[Sections~4.3--4.6 and Lemma~A.3]{MichettiProvenzanoSavo2026}. Their parameter-independent comparison coefficients and radial argument apply to every $s>0$, although stated for reduced flux $\nu\in(0,1/2]$. We use the parameters $\nu$ and $1-\nu$.
\begin{lemma}\label{lem:full-sphere-sector}
    Let $s>0$ and assume that the Gaussian curvature is at most one.
    \begin{enumerate}[label=\textup{(\roman*)},leftmargin=2.3em]
        \item If $\Omega$ is a one-pole surface with $|\Omega|\geq4\pi$, then
        \begin{equation*}
            \kappa_1(\Omega,A_p^{(s)})
            =\kappa_1(G_p;s)
            \leq s(s+1).
        \end{equation*}

        \item If $\Sigma$ is a closed two-pole surface, then
        \begin{equation*}
            w_1(\Sigma,A_{z_1,z_2}^{(s)})
            =w_1(G_{z_1,z_2};s)
            \leq s(s+1).
        \end{equation*}
    \end{enumerate}
    The same conclusions hold when the Gaussian curvature is at most $K>0$, with each right-hand side replaced by $Ks(s+1)$ and the area condition in \textup{(i)} replaced by $|\Omega|\geq4\pi/K$.
\end{lemma}
\begin{proof}
    We briefly recall the full-sphere comparison argument of \cite[Sections~4.3--4.6 and Lemma~A.3]{MichettiProvenzanoSavo2026}, using the notation appropriate to an arbitrary effective parameter $s>0$.
    
    Retain the unified notation $(X,G,\kappa_1(G;s))$ from \cref{sec:two-sector}, and set $M=|X|$. Thus $M\geq4\pi$ in both cases. Recall
    \begin{equation*}
        G^\star(a)=a(4\pi-a), \qquad 0<a<4\pi.
    \end{equation*}
    For sufficiently small $\varepsilon>0$, define
    \begin{equation*}
        a_\varepsilon=2\pi+\sqrt{4\pi^2-\varepsilon}, \qquad b_\varepsilon=2\pi-\sqrt{4\pi^2-\varepsilon},
    \end{equation*}
    so that $G^\star(a_\varepsilon)=G^\star(b_\varepsilon)=\varepsilon$. Let $R_\varepsilon\in(\pi/2,\pi)$ satisfy
    \begin{equation*}
        2\pi(1-\cos R_\varepsilon)=a_\varepsilon.
    \end{equation*}
    The geodesic cap of radius $R_\varepsilon$ in $\Sph$ therefore has area $a_\varepsilon$, and $R_\varepsilon\to\pi$ as $\varepsilon\to0$. Set
    \begin{equation*}
        \begin{aligned}
            G_\varepsilon^\star(a)&=
            \begin{cases}
                G^\star(a), & 0<a<a_\varepsilon,\\
                \varepsilon, & a_\varepsilon\leq a<M,
            \end{cases}
            &&\text{in the one-pole case},\\[0.4em]
            G_\varepsilon^\star(a)&=
            \begin{cases}
                G^\star(a), & 0<a<a_\varepsilon,\\
                \varepsilon, & a_\varepsilon\leq a\leq M-b_\varepsilon,\\
                G^\star(M-a), & M-b_\varepsilon<a<M,
            \end{cases}
            &&\text{in the two-pole case}.
        \end{aligned}
    \end{equation*}
    When $M=4\pi$, the middle interval in the second definition has zero length. The corresponding Green-level comparisons give $G\geq G_\varepsilon^\star$. Hence \cref{lem:monotonicity}, applied by approximation, yields
    \begin{equation}\label{eq:full-sphere-coefficient-comparison}
        \kappa_1(G;s)\leq\kappa_1(G_\varepsilon^\star;s).
    \end{equation}

    Let $\eta_{s,0}^D(R)$ be the first eigenvalue of $L_s$ on $(0,R)$ with the finite-energy condition at $r=0$ and the Dirichlet condition at $r=R$; see \cref{sec:cap-radial} for the definition and properties of this radial problem. Restricting the variational problem for $\kappa_1(G_\varepsilon^\star;s)$ to profiles vanishing on $(a_\varepsilon,M)$ gives
    \begin{equation*}
        \begin{aligned}
            \kappa_1(G_\varepsilon^\star;s) \leq \inf_{\substack{0\neq f\in\cF_{G_\varepsilon^\star}\\ f\equiv0\ \mathrm{on}\ (a_\varepsilon,M)}}\frac{\int_0^M\left( G_\varepsilon^\star f'^2+\frac{4\pi^2s^2}{G_\varepsilon^\star}f^2 \right)\dd a}{\int_0^M f^2\dd a}=\eta_{s,0}^D(R_\varepsilon).
        \end{aligned}
    \end{equation*}
    Here the last equality follows from $G_\varepsilon^\star=G^\star$ on $(0,a_\varepsilon)$ and the change of variables
    \begin{equation*}
        a=2\pi(1-\cos r).
    \end{equation*}
    Together with \eqref{eq:full-sphere-coefficient-comparison}, this gives
    \begin{equation*}
        \kappa_1(G;s)\leq\eta_{s,0}^D(R_\varepsilon).
    \end{equation*}

    Finally, the Dirichlet exhaustion argument in \cite[Lemma~A.3]{MichettiProvenzanoSavo2026}, applied to the radial problem with parameter $s$, gives
    \begin{equation*}
        \lim_{\varepsilon\to0}\eta_{s,0}^D(R_\varepsilon)=s(s+1).
    \end{equation*}
    The cutoff proof uses the profile $\sin^s r$ and requires only $s>0$. Letting $\varepsilon\to0$ proves both assertions. The curvature-$K$ statement follows by scaling.
\end{proof}

\begin{proof}[Proof of \cref{thm:C}]
    Apply \cref{lem:full-sphere-sector} with $s=\nu$ and $s=1-\nu$, and use \cref{prop:two-sector}. The resulting reciprocal-sum bounds are equivalent to the harmonic mean inequalities in \eqref{eq:large-intro} and \eqref{eq:closed-intro}.

    On $\SphK$ with antipodal poles, the separated spectrum computed in \cite[Proposition~A.1]{MichettiProvenzanoSavo2026} is
    \begin{equation*}
        K\bigl(|n-\nu|+k\bigr)\bigl(|n-\nu|+k+1\bigr), \qquad n\in\mathbb Z, \quad k\in\mathbb N_0.
    \end{equation*}
    For $0<\nu\leq1/2$, its first two eigenvalues, counted with multiplicity, are
    \begin{equation*}
        \lambda_1(\SphK,A_{p^*,-p^*}^{(\nu)})=K\nu(\nu+1), \qquad
        \lambda_2(\SphK,A_{p^*,-p^*}^{(\nu)})=K(1-\nu)(2-\nu).
    \end{equation*}
    This proves the model equalities in both statements.

    We now characterize equality and prove sharpness. By scaling, it suffices to take $K=1$.

    Suppose that equality holds in \eqref{eq:closed-intro}. Then both radial estimates must be equalities, and in particular
    \begin{equation*}
        w_1(G_{z_1,z_2};\nu)=\nu(\nu+1).
    \end{equation*}
    Choose a decreasing sequence $\varepsilon_j\to0$. If $|\Sigma|>4\pi$, then the strict part of \cref{lem:monotonicity} and the construction of $G_{\varepsilon_j}^\star$ give
    \begin{equation*}
        \begin{aligned}
            \nu(\nu+1) = w_1(G_{z_1,z_2};\nu) \leq w_1(G_{\varepsilon_1}^\star;\nu) < w_1(G_{\varepsilon_2}^\star;\nu) < \cdots < w_1(G_{\varepsilon_j}^\star;\nu) \leq \eta_{\nu,0}^D(R_{\varepsilon_j}) \to \nu(\nu+1),
        \end{aligned}
    \end{equation*}
    which is impossible. Therefore, $|\Sigma|=4\pi$. Since the Gaussian curvature is at most one, Gauss--Bonnet then implies that it is identically one, and hence $\Sigma$ is isometric to $\Sph$. Moreover, equality in the coefficient comparison gives
    \begin{equation*}
        G_{z_1,z_2}=G^\star,
    \end{equation*}
    which forces $z_1$ and $z_2$ to be antipodal. Conversely, the round sphere with antipodal poles attains equality.

    In the one-pole case, the same approximating coefficients satisfy
    \begin{equation*}
        \begin{aligned}
            \kappa_1(G_p;\nu) \leq \kappa_1(G_{\varepsilon_1}^\star;\nu) < \kappa_1(G_{\varepsilon_2}^\star;\nu) < \cdots < \kappa_1(G_{\varepsilon_j}^\star;\nu) \leq \eta_{\nu,0}^D(R_{\varepsilon_j}) \to \nu(\nu+1).
        \end{aligned}
    \end{equation*}
    Indeed, $\lim_{a\to M}G_p(a)>0$, since $a=M$ is a regular endpoint. Thus
    \begin{equation*}
        \kappa_1(G_p;\nu)<\nu(\nu+1),
    \end{equation*}
    and consequently
    \begin{equation*}
        \begin{aligned}
            \frac1{\lambda_1(\Omega,A_p^{(\nu)})}
            +\frac1{\lambda_2(\Omega,A_p^{(\nu)})}
            &\geq
            \frac1{\kappa_1(G_p;\nu)}
            +\frac1{\kappa_1(G_p;1-\nu)}\\
            &>
            \frac1{\nu(\nu+1)}
            +\frac1{(1-\nu)(2-\nu)}.
        \end{aligned}
    \end{equation*}
    Hence equality cannot occur in the boundary case. Scaling back gives the stated conclusions for arbitrary $K>0$.

    Finally, to prove sharpness in the boundary case, let centered geodesic caps exhaust $\SphK$ and rescale them to have area $4\pi/K$. Their Gaussian curvatures remain bounded above by $K$, while the scaling factors tend to one. By \cref{thm:cap-angular-ordering,lem:cap-exhaustion}, their first two magnetic eigenvalues converge to
    \begin{equation*}
        K\nu(\nu+1)
        \qquad\text{and}\qquad
        K(1-\nu)(2-\nu),
    \end{equation*}
    respectively. Therefore, the constant in \eqref{eq:large-intro} is sharp.
\end{proof}

\begin{corollary}
    Let $p^*,y\in\Sph$ be distinct and let $0<\nu\leq1/2$. Then
    \begin{equation*}
        \begin{aligned}
            \frac1{\lambda_1(\Sph,A_{p^*,y}^{(\nu)})}+\frac1{\lambda_2(\Sph,A_{p^*,y}^{(\nu)})} &\geq \frac1{\nu(\nu+1)}+\frac1{(1-\nu)(2-\nu)} \\
                                                                                            &= \frac1{\lambda_1(\Sph,A_{p^*,-p^*}^{(\nu)})}+\frac1{\lambda_2(\Sph,A_{p^*,-p^*}^{(\nu)})}.
        \end{aligned}
    \end{equation*}
    Equality holds if and only if $y=-p^*$.
\end{corollary}
\begin{proof}
    Apply \cref{thm:C}\textup{(ii)} with $K=1$.
\end{proof}

\subsection{Riemannian annuli}
\begin{proof}[Proof of \cref{thm:D}]
    Identify $C$ conformally with
    \begin{equation*}
        C_M= \mathbb{S}^{1}\times[-M,M], \qquad g_0=d\theta^2+dz^2,
    \end{equation*}
    and write its metric as $g=\rho^2g_0$, where $\rho$ is smooth and positive. By gauge invariance, we may assume
    \begin{equation*}
        A^{(\nu)}=\nu\,d\theta.
    \end{equation*}
    Consider the linearly independent trial functions
    \begin{equation}\label{eq:annulus-trials}
        \phi_0=1, \qquad \phi_1=\e^{\ii\theta}.
    \end{equation}
    Conformal invariance of the magnetic energy in dimension two and direct integration give
    \begin{equation*}
        \begin{aligned}
            \qform_{A^{(\nu)}}(\phi_0,\phi_0)=4\pi M\nu^2, \qquad \qform_{A^{(\nu)}}(\phi_1,\phi_1)=4\pi M(1-\nu)^2, \qquad \qform_{A^{(\nu)}}(\phi_0,\phi_1)=0.
        \end{aligned}
    \end{equation*}
    Moreover, since $|\phi_0|=|\phi_1|=1$,
    \begin{equation*}
        \|\phi_0\|_{L^2(C,g)}^2=\|\phi_1\|_{L^2(C,g)}^2=|C|.
    \end{equation*}
    Therefore, \cref{lem:trace-ritz} yields
    \begin{equation*}
        \begin{aligned}
            \frac1{\lambda_1(C,A^{(\nu)})}+\frac1{\lambda_2(C,A^{(\nu)})}\geq\tr(\Kmat^{-1}\Mmat)=\frac{|C|}{4\pi M}\left(\frac1{\nu^2}+\frac1{(1-\nu)^2}\right),
        \end{aligned}
    \end{equation*}
    which is equivalent to \eqref{eq:annulus-two-mode-intro}.

    On $C_M$, the magnetic Laplacian in the coordinates $(\theta,z)$ is
    \begin{equation*}
        \Delta_{A^{(\nu)},g_0}u=-u_{zz}-u_{\theta\theta}+2\ii\nu u_\theta+\nu^2u.
    \end{equation*}
    Writing $u(\theta,z)=\e^{\ii n\theta}v(z)$, with $n\in\mathbb Z$, reduces the eigenvalue problem to
    \begin{equation*}
        \begin{cases}
            \ -v''(z)+|n-\nu|^2v(z)=\lambda v(z), & -M<z<M,\\
            \ v'(-M)=v'(M)=0.
        \end{cases}
    \end{equation*}
    Here the magnetic Neumann condition reduces to the ordinary one because $A^{(\nu)}(\partial_z)=0$. The axial eigenfunctions are, up to normalization,
    \begin{equation*}
        v_k(z)=\cos\left(\frac{k\pi(z+M)}{2M}\right), \qquad k\in\mathbb N_0.
    \end{equation*}
    The products $\e^{\ii n\theta}v_k(z)$ form a complete eigenbasis, so the spectrum, counted with multiplicity, is
    \begin{equation*}
        \lambda_{n,k}(C_M,A^{(\nu)})=|n-\nu|^2+\frac{k^2\pi^2}{4M^2}, \qquad n\in\mathbb Z,\quad k\in\mathbb N_0.
    \end{equation*}
    Consequently, for $0<\nu<1/2$,
    \begin{equation*}
        \lambda_1(C_M,A^{(\nu)})=\nu^2, \qquad \lambda_2(C_M,A^{(\nu)})=\min\left\{(1-\nu)^2,\nu^2+\frac{\pi^2}{4M^2}\right\}.
    \end{equation*}
    Hence $\e^{\ii\theta}$ belongs to the second eigenspace precisely when \eqref{eq:cylinder-threshold-intro} holds. At $\nu=1/2$, the functions $1$ and $\e^{\ii\theta}$ span the eigenspace associated with the first eigenvalue $1/4$, for every $M>0$. Scaling $C_M$ to area $|C|$ multiplies its reciprocal eigenvalues by $|C|/(4\pi M)$, so \eqref{eq:annulus-two-mode-intro} yields \eqref{eq:annulus-model-intro}.

    We now characterize equality under \eqref{eq:cylinder-threshold-intro}. If equality holds in \eqref{eq:annulus-model-intro}, the equality statement in \cref{lem:trace-ritz} implies that
    \begin{equation*}
        V=\Span_{\mathbb C}\{1,\e^{\ii\theta}\}
    \end{equation*}
    is invariant under $\Delta_{A^{(\nu)},g}$. Since
    \begin{equation*}
        \Delta_{A^{(\nu)},g}1
        =\rho^{-2}\Delta_{A^{(\nu)},g_0}1
        =\nu^2\rho^{-2},
    \end{equation*}
    invariance and $\nu>0$ imply $\rho^{-2}\in V$. The only real-valued functions in $V$ are constants, so $\rho$ is constant and $C$ is homothetic to $C_M$.

    Conversely, if $C$ is homothetic to $C_M$ and \eqref{eq:cylinder-threshold-intro} holds, the trial functions in \eqref{eq:annulus-trials} realize the first two eigenvalues, counted with multiplicities, and equality follows.
\end{proof}
If \eqref{eq:cylinder-threshold-intro} fails, the second eigenvalue in the $n=0$ angular mode is smaller than the first eigenvalue in the $n=1$ angular mode. The estimate \eqref{eq:annulus-two-mode-intro} remains valid, but its upper bound is then strictly larger than $H_{1,2}(\widehat C_M,A^{(\nu)})$.

\appendix
\section{Spherical cap ordering and sharpness}\label{sec:cap-appendix}
We prove the cap-ordering theorem used in \cref{sec:caps} and establish the sharpness of \cref{thm:C}\textup{(i)}. Throughout this section, the sphere has curvature one; the corresponding statements for $\SphK$ follow by scaling.

\subsection{Radial problems on a spherical cap}\label{sec:cap-radial}
Let $\Omega_R^\star\subset\Sph$ be the geodesic disk of radius $R\in(0,\pi)$ centered at $p^*$. In polar coordinates $(r,\theta)$ around $p^*$,
\begin{equation*}
    g=dr^2+\sin^2r\,d\theta^2, \qquad A_{p^*}^{(\nu)}=\nu\,d\theta.
\end{equation*}
For $u(r,\theta)=v(r)\e^{\ii n\theta}$, $n\in\mathbb Z$, the effective angular parameter is $s=|n-\nu|$, and
\begin{equation*}
    \int_{\Omega_R^\star}|d^{A_{p^*}^{(\nu)}}u|^2\dv=2\pi\int_0^R\left(|v'|^2\sin r+\frac{s^2}{\sin r}|v|^2\right)\dd r.
\end{equation*}
For $s\geq0$, consider
\begin{equation*}
    L_sv=-\frac1{\sin r}(\sin r v')'+\frac{s^2}{\sin^2r}v
\end{equation*}
in $L^2((0,R),\sin r\dd r)$. Its Neumann form domain is
\begin{equation*}
    \mathcal Q_0=\left\{v\in H^1_{\mathrm{loc}}(0,R):\int_0^R(|v'|^2+|v|^2)\sin r\dd r<\infty\right\}
\end{equation*}
when $s=0$, whereas for $s>0$ it is
\begin{equation*}
    \mathcal Q_+=\left\{v\in\mathcal Q_0:\int_0^R\frac{|v|^2}{\sin r}\dd r<\infty\right\}.
\end{equation*}
The corresponding quadratic form is
\begin{equation*}
    q_s[v]=\int_0^R\left(|v'|^2\sin r+s^2\frac{|v|^2}{\sin r}\right)\dd r.
\end{equation*}
For $s\geq0$, let $\eta_{s,j}^N(R)$, $j\geq0$, denote the eigenvalues, in strictly increasing order, of the Sturm--Liouville problem
\begin{equation*}
    \begin{cases}
        \ -(\sin r\,v')'+\dfrac{s^2}{\sin r}v = \eta\sin r\,v, & 0<r<R,\\[0.4em]
        \ \displaystyle\lim_{r\to 0}\sin r\,v'(r)=0, \qquad v'(R)=0.
    \end{cases}
\end{equation*}
The condition at $r=0$ selects the finite-energy solution in $\mathcal Q_0$ when $s=0$ and in $\mathcal Q_+$ when $s>0$. In particular, $\eta_{0,0}^N(R)=0$, and for $s>0$,
\begin{equation*}
    \kappa_1(\Omega_R^\star,A_{p^*}^{(s)})=\eta_{s,0}^N(R).
\end{equation*}

Similarly, for $s>0$, let $\eta_{s,j}^D(R)$, $j\geq0$, denote the eigenvalues, in strictly increasing order, of
\begin{equation*}
    \begin{cases}
        \ -(\sin r\,v')'+\dfrac{s^2}{\sin r}v = \eta\sin r\,v, & 0<r<R,\\[0.4em]
        \ \displaystyle\lim_{r\to0}\sin r\,v'(r)=0, \qquad v(R)=0.
    \end{cases}
\end{equation*}
The Dirichlet quadratic form is $q_s$ restricted to
\begin{equation*}
    \left\{v\in\mathcal Q_+:v(R)=0\right\}.
\end{equation*}
Both radial problems have simple eigenvalues.

Separation of variables gives
\begin{equation}\label{eq:cap-separated-spectrum}
    \operatorname{spec}(\Omega_R^\star,A_{p^*}^{(\nu)})=\biguplus_{n\in\mathbb Z}\biguplus_{j\geq0}\left\{\eta_{|n-\nu|,j}^N(R)\right\},
\end{equation}
where the union counts multiplicities.
\begin{lemma}\label{lem:cap-angular-monotonicity}
    If $0<s<t$, then $\eta_{s,0}^N(R)<\eta_{t,0}^N(R)$.
\end{lemma}
\begin{proof}
    Let $v_t$ be a positive normalized first Neumann eigenfunction of $L_t$. Since $v_t\in\mathcal Q_+$,
    \begin{equation*}
        \eta_{s,0}^N(R)\leq q_s[v_t] = q_t[v_t]-(t^2-s^2)\int_0^R\frac{v_t^2}{\sin r}\dd r < \eta_{t,0}^N(R).
    \end{equation*}
\end{proof}

\subsection{The first two cap eigenvalues}
To compare the higher radial eigenvalues, we first relate the radial Neumann and Dirichlet spectra.
\begin{lemma}\label{lem:radial-ND-shift}
    For every $j\geq0$,
    \begin{equation}\label{eq:radial-ND-shift}
        \eta_{0,j+1}^N(R)=\eta_{1,j}^D(R).
    \end{equation}
    Moreover,
    \begin{equation}\label{eq:strict-DN}
        \eta_{1,0}^N(R)<\eta_{1,0}^D(R).
    \end{equation}
\end{lemma}
\begin{proof}
    In $L^2((0,R),\sin r\dd r)$, consider
    \begin{equation*}
        T=\frac{d}{dr}, \qquad T^*=-\frac{d}{dr}-\cot r,
    \end{equation*}
    where $T^*$ is the formal adjoint of $T$. Direct calculation gives
    \begin{equation*}
        L_0=T^*T, \qquad L_1=TT^*.
    \end{equation*}
    Let $f$ be a Neumann eigenfunction of $L_0$ with eigenvalue $\eta>0$. The factorization shows that $g=Tf=f'$ satisfies
    \begin{equation*}
        L_1g=T(L_0f)=\eta g, \qquad g(R)=0.
    \end{equation*}
    Conversely, if $g$ is a Dirichlet eigenfunction of $L_1$ with eigenvalue $\eta$, then
    \begin{equation*}
        f=\frac1\eta T^*g=-\frac1\eta(g'+\cot r\,g)
    \end{equation*}
    satisfies $Tf=g$, $L_0f=\eta f$, and $f'(R)=0$. The regular asymptotics at $r=0$ show that both maps preserve the finite-energy condition. They are inverse on the corresponding eigenspaces, so the positive Neumann spectrum of $L_0$ coincides with the Dirichlet spectrum of $L_1$. Since $\eta_{0,0}^N(R)=0$, ordering the eigenvalues proves \eqref{eq:radial-ND-shift}.

    The inclusion of the Dirichlet form domain in the Neumann form domain gives
    \begin{equation*}
        \eta_{1,0}^N(R)\leq\eta_{1,0}^D(R).
    \end{equation*}
    If equality held, a first Dirichlet eigenfunction would also minimize the Neumann Rayleigh quotient and hence satisfy $g'(R)=0$. Since $g(R)=0$, uniqueness for the radial equation at the regular endpoint $R$ would then imply $g\equiv0$, a contradiction. This proves \eqref{eq:strict-DN}.
\end{proof}

\begin{proof}[Proof of \cref{thm:cap-angular-ordering}]
    Suppose first that $0<\nu<1/2$. The two smallest effective angular parameters are $\nu$ and $1-\nu$, corresponding to $n=0$ and $n=1$, respectively. By \cref{lem:cap-angular-monotonicity},
    \begin{equation*}
        \eta_{\nu,0}^N(R)<\eta_{1-\nu,0}^N(R),
    \end{equation*}
    and the first radial eigenvalue in every other angular mode is strictly larger than $\eta_{1-\nu,0}^N(R)$.

    It remains to compare this value with the second radial eigenvalue in the $n=0$ angular mode. Since $\mathcal Q_+\subset\mathcal Q_0$ and $q_\nu[v]\geq q_0[v]$ on $\mathcal Q_+$, the min--max principle gives
    \begin{equation*}
        \begin{aligned}
            \eta_{\nu,1}^N(R) &= \inf_{\substack{V\subset\mathcal Q_+\\ \dim V=2}}\sup_{0\neq v\in V}\frac{q_\nu[v]}{\int_0^R|v|^2\sin r\dd r}\\
                              &\geq \inf_{\substack{V\subset\mathcal Q_+\\ \dim V=2}}\sup_{0\neq v\in V}\frac{q_0[v]}{\int_0^R|v|^2\sin r\dd r}\\
                              &\geq \inf_{\substack{V\subset\mathcal Q_0\\ \dim V=2}}\sup_{0\neq v\in V}\frac{q_0[v]}{\int_0^R|v|^2\sin r\dd r} = \eta_{0,1}^N(R).
        \end{aligned}
    \end{equation*}
    Together with \cref{lem:cap-angular-monotonicity} and \cref{lem:radial-ND-shift}, this yields
    \begin{equation*}
        \eta_{\nu,1}^N(R)\geq\eta_{0,1}^N(R)=\eta_{1,0}^D(R)>\eta_{1,0}^N(R)>\eta_{1-\nu,0}^N(R).
    \end{equation*}
    Thus the first two entries of the separated spectrum \eqref{eq:cap-separated-spectrum} are $\eta_{\nu,0}^N(R)$ and $\eta_{1-\nu,0}^N(R)$. Each occurs in exactly one angular mode and is simple in the corresponding radial problem, proving \textup{(i)}.

    If $\nu=1/2$, the angular modes $n=0$ and $n=1$ have the same effective parameter $1/2$. The simple first radial eigenvalue therefore gives two linearly independent magnetic eigenfunctions. Every other angular mode has a strictly larger first radial eigenvalue by \cref{lem:cap-angular-monotonicity}, and all higher radial eigenvalues are larger still. This proves \textup{(ii)}.
\end{proof}

\begin{remark}
    For fixed $\nu$ and $n\in\mathbb Z$, a radial eigenfunction $v$ associated with $\eta_{|n-\nu|,j}^N(R)$ gives the magnetic Neumann eigenfunction $v(r)\e^{\ii n\theta}$ with the same eigenvalue. Arranging the values in \eqref{eq:cap-separated-spectrum} in nondecreasing order, with multiplicities, gives $\lambda_k(\Omega_R^\star,A_{p^*}^{(\nu)})$, $k\geq1$. Thus $j$ indexes the eigenvalues within a fixed angular mode, whereas $k$ indexes the full magnetic spectrum.
\end{remark}

\subsection{The full-sphere limit and sharpness}
For the sharpness assertion in \cref{rem:punctured-model}, we determine the limit of the first radial Neumann eigenvalue for each positive parameter.
\begin{lemma}\label{lem:cap-exhaustion}
    For every $s>0$,
    \begin{equation*}
        \lim_{R\to\pi}\eta_{s,0}^N(R)=s(s+1).
    \end{equation*}
\end{lemma}
\begin{proof}
    Fix $s>0$. The function $f(r)=\sin^s r$ has finite energy on $(0,\pi)$ and satisfies
    \begin{equation*}
        L_sf=s(s+1)f.
    \end{equation*}
    Since $f$ is positive, it is a first eigenfunction of the full-sphere radial problem with finite-energy conditions at both endpoints.
    The restriction of $f$ to $(0,R)$ belongs to the Neumann form domain $\mathcal Q_+$. Hence
    \begin{equation*}
        \eta_{s,0}^N(R)\leq\frac{\int_0^R\left(|f'|^2\sin r+\frac{s^2}{\sin r}|f|^2\right)\dd r}{\int_0^R|f|^2\sin r\dd r}.
    \end{equation*}
    Both integrals converge to their full-sphere counterparts as $R\to\pi$, so
    \begin{equation}\label{eq:cap-exhaustion-upper}
        \limsup_{R\to\pi}\eta_{s,0}^N(R)\leq s(s+1).
    \end{equation}

    For the reverse inequality, choose $R_j\to\pi$ realizing the lower limit of the first radial eigenvalues, and let $v_j$ be positive normalized first eigenfunctions. By \eqref{eq:cap-exhaustion-upper}, there is a constant $C$ independent of $j$ such that
    \begin{equation*}
        \int_0^{R_j}v_j^2\sin r\dd r=1, \qquad \int_0^{R_j}\left(|v_j'|^2\sin r+\frac{s^2}{\sin r}v_j^2\right)\dd r = \eta_{s,0}^N(R_j)\leq C.
    \end{equation*}
    These estimates give uniform $H^1$ bounds on compact subintervals of $(0,\pi)$. A diagonal subsequence yields a function $v$ such that $v_j$ converges weakly in $H^1$ and strongly in $L^2$ on each such subinterval.

    We next show that $v$ has unit norm. For $0<\delta<\pi/2$ and all sufficiently large $j$, one has $R_j>\pi-\delta$ and
    \begin{equation*}
        \begin{aligned}
            \int_0^\delta v_j^2\sin r\dd r+\int_{\pi-\delta}^{R_j}v_j^2\sin r\dd r\leq\sin^2\delta\int_0^{R_j}\frac{v_j^2}{\sin r}\dd r\leq\frac{C\sin^2\delta}{s^2}.
        \end{aligned}
    \end{equation*}
    Thus the squared norm near the poles tends uniformly to zero as $\delta\to0$. Strong convergence on $(\delta,\pi-\delta)$ gives
    \begin{equation*}
        1-\frac{C\sin^2\delta}{s^2} \leq \int_\delta^{\pi-\delta}|v|^2\sin r\dd r \leq 1.
    \end{equation*}
    Letting $\delta\to0$ gives
    \begin{equation*}
        \int_0^\pi|v|^2\sin r\dd r=1.
    \end{equation*}

    Finally, lower semicontinuity on compact subintervals, followed by $\delta\to0$, gives
    \begin{equation*}
        \int_0^\pi\left(|v'|^2\sin r+\frac{s^2}{\sin r}|v|^2\right)\dd r \leq \liminf_{j\to\infty}\eta_{s,0}^N(R_j).
    \end{equation*}
    Hence $v$ is an admissible normalized competitor for the full-sphere radial problem, whose variational characterization gives
    \begin{equation*}
        s(s+1) \leq \int_0^\pi\left(|v'|^2\sin r+\frac{s^2}{\sin r}|v|^2\right)\dd r \leq \liminf_{R\to\pi}\eta_{s,0}^N(R).
    \end{equation*}
    Together with \eqref{eq:cap-exhaustion-upper}, this proves the assertion.
\end{proof}

To prove sharpness in \cref{rem:punctured-model}, rescale the metric of each centered cap $\Omega_R^\star$ by $4\pi/|\Omega_R^\star|$. The rescaled caps have area $4\pi$ and Gaussian curvature $|\Omega_R^\star|/(4\pi)<1$.
The scaling factors tend to one as $R\to\pi$, so \cref{thm:cap-angular-ordering,lem:cap-exhaustion} give
\begin{equation*}
    \begin{aligned}
        \lim_{R\to\pi}\frac{|\Omega_R^\star|}{4\pi}\lambda_1(\Omega_R^\star,A_{p^*}^{(\nu)})=\nu(\nu+1),\qquad\qquad
        \lim_{R\to\pi}\frac{|\Omega_R^\star|}{4\pi}\lambda_2(\Omega_R^\star,A_{p^*}^{(\nu)})=(1-\nu)(2-\nu).
    \end{aligned}
\end{equation*}
Their harmonic means therefore converge to the right-hand side of \eqref{eq:large-intro} with $K=1$, proving sharpness. The conclusion for every $K>0$ follows by scaling.

\section*{Acknowledgments}
The authors acknowledge the use of AI for assistance with English-language editing and polishing. 
The authors take full responsibility for the mathematical content and the final manuscript.

\bibliographystyle{amsalpha}
\bibliography{CY-MagLap-2609.bib}

\vspace{1cm}

\end{document}